\documentclass[11pt]{article}
\usepackage{graphicx}
\usepackage{amsfonts,amsmath,amssymb,amscd,enumitem,color}
\usepackage{mathtools}
\usepackage{accents}
\usepackage{amsthm}
\usepackage[hidelinks]{hyperref}
\usepackage[capitalize]{cleveref}
\usepackage[english]{babel}
\usepackage[utf8]{inputenc}
 
\newtheorem{theorem}{Theorem}
\newtheorem{proposition}[theorem]{Proposition}
\newtheorem{definition}[theorem]{Definition}
\newtheorem{lemma}[theorem]{Lemma}

\newtheorem{conjecture}[theorem]{Conjecture}

\numberwithin{subcase}{case}

\numberwithin{subsubcase}{subcase}

\newcommand{\poly}{\operatorname{poly}}

\usepackage{fullpage}

\author{Brett Leroux, Luis Rademacher}

\def\keywords#1{\par\addvspace\medskipamount{\rightskip=0pt plus1cm
\def\and{\ifhmode\unskip\nobreak\fi\ $\cdot$
}\noindent{Keywords:}\enspace\ignorespaces#1\par}}

\title{Expansion of random $0/1$ polytopes}

\usepackage{dirtytalk}

\def\final{0}  

\ifnum\final=0  
\newcommand{\lnote}[1]{[{\small Luis: \bf #1}]}
\newcommand{\bnote}[1]{[{\small Brett: \bf #1}]}
\newcommand{\anonnote}[1]{[{\small anon: \bf #1}]}
\newcommand{\sidecomment}[1]{\marginpar{\tiny #1}}
\newcommand{\details}[1]{{\color{blue}\ [[#1]] }}
\else 
\newcommand{\lnote}[1]{}
\newcommand{\bnote}[1]{}
\newcommand{\anonnote}[1]{}
\newcommand{\sidecomment}[1]{}
\newcommand{\details}[1]{}
\fi  

\newcommand{\Rl}{\operatorname{\mathbb{R}}}

\newcommand{\Tor}{\operatorname{\Tor}}

\newcommand{\aff}{\operatorname{aff}}
\newcommand{\conv}{\operatorname{conv}}

\newcommand{\suchthat}{\mathrel{:}}

\makeatletter
\newcommand\fake@math{}
\def\fake@math#1\){[math]}
\makeatother

\begin{document}
\maketitle

\begin{abstract}
    A conjecture of Mihail and Vazirani \cite{FederMihail} states that the edge expansion of the graph of every $0/1$ polytope is at least one. Any lower bound on the edge expansion gives an upper bound for the mixing time of a random walk on the graph of the polytope. Such random walks are important because they can be used to generate an element from a set of combinatorial objects uniformly at random. A weaker form of the conjecture of Mihail and Vazirani says that the edge expansion of the graph of a $0/1$ polytope in $\Rl^d$ is greater than 1 over some polynomial function of $d$. This weaker version of the conjecture would suffice for all applications. Our main result is that the edge expansion of the graph of a \emph{random} $0/1$ polytope in $\Rl^d$ is at least $\frac{1}{12d}$ with high probability. 
\end{abstract}

\section{Introduction} 

A \emph{$0/1$ polytope} (in $\Rl^d$) is the convex hull of some subset of $\{0,1\}^d$. 
In other words, $0/1$ polytopes are polytopes such that every coordinate of every vertex is either 0 or 1. 
One reason these polytopes have been studied is their connection to various combinatorial optimization problems. 
This connection arises due to the fact that many combinatorial structures can be described by a set of $0/1$ vectors. For example, if $M$ is a matroid whose ground-set has size $d$, then every basis of $M$ corresponds to a $0/1$ vector in $\{0,1\}^d$. One can then define the \emph{matroid base polytope} of $M$ as the convex hull of the $0/1$ vectors corresponding to bases of $M$. With this construction, questions about the combinatorial structure of $M$ can be restated as questions about the geometric structure of the matroid base polytope of $M$. See \cite{MR183532,MR2163945} for an early example of the use of this idea.

The applications of $0/1$ polytopes that are most relevant to this paper depend only on the graph of the polytope. For a polytope $P$, the graph $G(P)$ of $P$ is the graph whose vertices are vertices of $P$ and whose edges are edges of $P$. It turns out that by performing a random walk on the graph of a $0/1$ polytope, one can solve a number of important combinatorial optimization problems. The prime example of this is the problem of sampling from a set of combinatorial objects uniformly at random. In our setting, the set of combinatorial objects naturally corresponds to the set of vertices of some $0/1$ polytope. Thus, the problem of generating such a random sample is reduced to the problem of generating a random vertex of a $0/1$ polytope. 
This can be done efficiently as long as the random walk on the graph of the polytope mixes rapidly, i.e., approaches the stationary distribution in $\poly(d)$ steps. 
This rapid mixing can be guaranteed to occur if one can obtain a $1/\poly(d)$ lower bound on a quantity associated to the graph called the \emph{edge expansion}. This is well known, see for example \cite{DBLP:conf/random/KaibelR03,Mihail}. We explain in more detail the relationship between edge expansion and rapid mixing below. First, we define edge expansion.

For a graph $G = (V,E)$, and a subset $S \subset V$, we use $\delta(S)$ to denote the set of edges that connect a vertex in $S$ to a vertex in $V \setminus S$. With this, we can the define the edge expansion of a graph as follows

\begin{definition}
The \emph{edge expansion} of a graph $G = (V,E)$ is
\[
\min\bigg\{\frac{|\delta(S)|}{|S|} \suchthat S \subset V, 1 \le |S| \le \frac{|V|}{2}\bigg\}.
\]
\end{definition}
Similarly, the edge expansion of a polytope $P$ is defined to be the edge expansion of the graph $G(P)$ of $P$. 

The proof that a good lower bound on edge expansion implies rapid mixing is roughly as follows. A lower bound on the edge expansion implies, by the Cheeger inequality for general graphs as stated in \cite{MR1395858}, a lower bound on the spectral gap of the Laplacian of the graph. It is then a standard fact that a lower bound on the spectral gap implies rapid mixing, see for example \cite{MR1211324}.

The main motivation for this paper is the conjecture of Mihail and Vazirani which states that all $0/1$ polytopes have edge expansion at least 1. See \cite[Section 7]{FederMihail} and \cite{Mihail}. For applications, it would suffice to establish the following weaker form of Mihail and Vazirani's conjecture which has been mentioned in a number of previous works including \cite{MR2231088, MR2077562, Mihail}.

\begin{conjecture}\label{conj:poly}
The edge expansion of the graph of a $0/1$ polytope in $\Rl^d$ is greater than $\frac{1}{f(d)}$ for some polynomial function $f$. 
\end{conjecture}

As mentioned above, a proof of this conjecture would have important applications to the analysis of randomized algorithms for combinatorial problems. For details concerning such applications, see \cite{FederMihail,Gillmann,MR2077562, Mihail}.

A number of previous works have made some progress on the above conjecture, by establishing it for various special classes of $0/1$ polytopes. We overview such previous work in \cref{sec:prev}. As another special case, it was asked in \cite{DBLP:conf/random/KaibelR03} and \cite{Gillmann} whether \emph{random} $0/1$ polytopes have good expansion properties. Our main result gives an affirmative answer to this question. We consider three different (but similar) models of random $0/1$ polytopes which we call the \emph{balls-into-bins} model, the \emph{binomial} model and the \emph{uniform} model. 
See the next section for definitions of these models. 
We prove that the edge expansion of a random $0/1$ polytope distributed according to any of these three models is at least $1/12d$ with high probability.
In this theorem and everywhere else in the paper the phrase ``with high probability'' means ``with probability lower bounded by a function of $d$ alone that converges to 1 as $d$ goes to $\infty$.''

\begin{theorem}\label{thm:main}
Assume that $P \subset \Rl^d$ is a random $0/1$ polytope that is distributed according to either the balls-into-bins model, the binomial model, or the uniform model as defined in \cref{sec:random}. Then the edge expansion of $P$ is at least $1/12d$ with high probability.  
\end{theorem}

See \cref{sec:proofs} for the proof of this theorem. 

A rough idea of the proof is as follows: Say we have a random $0/1$ polytope $P$ in $\Rl^d$ with $n$ vertices. It is possible to choose an integer $k$ which depends on $n,d$ such that if we consider the orthogonal projection of $P$ to the first $k$ coordinates, then the projected vertices of $P$ cover the vertices of the $k$-cube $C^k$ in the projected space and also the projected vertices of $P$ are well distributed among the vertices of $C^k$ in the sense that not too many vertices of $P$ are projected to the same vertex of $C^k$. We then use the fact that $C^k$ has good edge expansion to show that $P$ also must have good edge expansion. Apart from being interesting in their own right, these results provide some evidence that the above weaker form of the conjecture (\cref{conj:poly}) of Mihail and Vazirani may be true. 


\section{Models of randomness}\label{sec:random}

In this section we introduce the models of random $0/1$ polytopes that we consider. 

The most familiar example of a $0/1$ polytope in $\Rl^d$ is, of course, the regular $d$-dimensional cube. We use the notation $C^d: = [0,1]^d$ for the regular $d$-dimensional cube in $\Rl^d$. The vertex set of the cube $C^d$ is $\{0,1\}^d$ and so every $0/1$ polytope can be seen as the convex hull of some subset of vertices of $C^d$ for some $d$. Therefore, to generate a random $0/1$ polytope, one can first pick some random subset $S \subset \{0,1\}^d$ and then form the polytope by taking the convex hull of $S$. For a set $S \subset \{0,1\}^d$, we use $\conv S$ to denote the convex hull of $S$, i.e., the $0/1$ polytope with vertex set $S$.

We restrict our attention to the following three models of random $0/1$ polytopes.

\begin{enumerate}
    \item \textbf{The \emph{balls-into-bins} model:} For any $n \in \mathbb{N}$, choose $S_1, \dotsc, S_n$ independently and uniformly from $\{0,1\}^d$. Repetition is allowed. Define the set $S_n^d := \{S_1,\dotsc, S_n\}$ and the polytope $P_n^d  := \conv S_n^d$. 
    \item \textbf{The \emph{binomial} model:} For any $p \in (0,1)$, let $S_p^d$ be the subset of $\{0,1\}^d$ where each $v \in \{0,1\}^d$ is in $S_p^d$ with probability $p$. Define the polytope $P_p^d := \conv S_p^d$.
    \item \textbf{The \emph{uniform} model:} For any $1 \le n \le 2^d$, let $U_{n}^d$ be chosen uniformly at random from the set of all $n$-element subsets of $\{0,1\}^d$. Define the polytope $Q_{n}^d:= \conv U_{n}^d$.
\end{enumerate}

\section{Previous work on expansion of \texorpdfstring{$0/1$}{0/1} polytopes}\label{sec:prev}
Some important families of $0/1$ polytopes are known to have edge expansion at least 1. We give an overview of what is known below. We also explain what is known about a closely related expansion property called \emph{vertex expansion} (defined in \cref{sec:vert}).

\subsection{Edge expansion}
In a recent breakthrough, the authors of \cite{MR4003314} showed that the matroid base polytope of any matroid has edge expansion at least one (\cite[Theorem 1.5]{MR4003314}). That is, they established the original conjecture of Mihail and Vazirani (that all $0/1$ polytopes have edge expansion at least one) for $0/1$ polytopes which are the matroid base polytope of some matroid. 

Prior to this breakthrough, the conjecture had only been established for some more limited families of $0/1$ polytopes: Kaibel showed in \cite{MR2077562} that the conjecture holds for $0/1$ polytopes of dimension at most five, simple $0/1$ polytopes, hypersimplices, stable set polytopes, and perfect matching polytopes. In earlier papers, the conjecture had been established for matching polytopes, order ideal polytopes, and independent set polytopes in \cite{mmm}, and for balanced matroid base polytopes in \cite{FederMihail}.

Despite this progress, the conjecture of Mihail and Vazirani is still a long way from being fully solved. Indeed, most $0/1$ polytopes do not fall into any of the categories mentioned above. Some examples of $0/1$ polytopes for which the conjecture still open are knapsack polytopes, equality constrained $0/1$ polytopes \cite{DBLP:journals/dam/MatsuiT95}, and symmetric
traveling salesman polytopes. For these polytopes, the weaker form of the conjecture, (\cref{conj:poly}), is also still open.

\subsection{Vertex expansion}\label{sec:vert}
There is another notion of expansion called \emph{vertex expansion}. The vertex expansion is relevant to our considerations because the vertex expansion of a graph is a lower bound on the edge expansion of the graph. For a graph $G = (V,E)$, and a subset $S \subset V$, we use $N(S)$ to denote the set of all $v \in V \setminus S$ such that there is an edge connecting $v$ to some $s \in S$. With this, we can define vertex expansion as follows

\begin{definition}
The \emph{vertex expansion} of a graph $G = (V,E)$ is
\[
\min\bigg\{\frac{|N(S)|}{|S|} \suchthat S \subset V, 1 \le |S| \le \frac{|V|}{2}\bigg\}.
\]
\end{definition}
The vertex expansion of a polytope is the vertex expansion of the graph of the polytope. 

Because the vertex expansion gives a lower bound on the edge expansion, in the context of \cref{conj:poly}, it is natural to ask whether one can establish a $1/\poly(d)$ lower bound on the vertex expansion of $0/1$ polytopes. Unfortunately, this is known to be impossible: Gillmann showed in his thesis \cite{Gillmann} that there exists a sequence $\{P_d\}_{d \in \mathbb{N}}$ of $0/1$ polytopes $P_d$ in $\Rl^d$ such that the vertex expansion of $P_d$ is at most $2^{-.32192d}$ for $d$ sufficiently large.\footnote{A similar construction was mentioned in \cite{Mihail}, but it seems that the details were never published.}

In the construction of the polytopes $P_d$, some of the vertices are chosen deterministically and some are chosen randomly. In contrast to Gillman's result, we can show that if the vertices are chosen completely randomly, then the polytope will have $1/\poly(d)$ vertex expansion with high probability. In particular, we can prove that the vertex expansion of a random $0/1$ polytope distributed according to any of the three models described in \cref{sec:random} is $\Omega(1/d^{3/2})$ with high probability: 

\begin{theorem}\label{thm:vertex}
Assume that $P \subset \Rl^d$ is a random $0/1$ polytope that is distributed according to either the balls-into-bins model, the binomial model, or the uniform model as defined in \cref{sec:random}. 
Then the vertex expansion of $P$ is $\Omega(1/d^{3/2})$ with high probability.  
\end{theorem}

The proof of the above theorem is nearly the same as the proof of the corresponding result for edge expansion (i.e. \cref{thm:main}) and is thus omitted. Whereas in the proof of \cref{thm:main} we use the fact that the edge expansion of the $d$-dimensional cube $C^d$ is 1, in the proof of \cref{thm:vertex} one uses the well known fact that the vertex expansion of the $d$-dimensional cube $C^d$ is $\Omega(1/\sqrt{d})$. This fact is sometimes called Harper's theorem, see \cite{MR200192}.

\section{Background on polytopes}
Previous works which established good edge expansion for special classes of $0/1$ polytopes used mainly combinatorial proof techniques. Our approach, in contrast, is purely geometric. Thus, we need some basic facts about the geometry of convex polytopes.

As is standard, by a \emph{polytope} we always mean a \emph{convex polytope} and we sometimes omit the word convex. We refer the reader to \cite{Ziegler} for a comprehensive introduction to the theory of convex polytopes and to \cite{MR1785291} for a survey on $0/1$ polytopes in particular. 

Let $P \subset \Rl^d$ be a polytope. A \emph{face} $F$ of $P$ is any set that can be written as $F = \{x \in P \suchthat c \cdot x = c_0\}$ where $c \cdot x \le c_0$ is some linear inequality that is satisfied by all $x \in P$. A \emph{proper face} of $P$ is any face of $P$ which is not equal to either $P$ or $\emptyset$. For a polytope $P$, we use the notation $V(P)$ for the set of vertices of $P$, i.e., the set of $0$-dimensional faces of $P$ and $E(P)$ for the set of edges, i.e. the set of $1$-dimensional faces of $P$. 

Aside from these definitions, the only fact about polytopes we need is the following basic result that is often used without proof. We give a proof for the sake of completeness. 
\begin{proposition}\label{prop:poly}
If $P \subset \Rl^d$ is a $d$-polytope (i.e. $P$ is full-dimensional, so that $\aff P = \Rl^d$), then for any vertex $v$ of $P$, the set of edges incident to $v$ are not contained in any hyperplane.
\end{proposition}
\begin{proof}
By \cite[Proposition 2.4]{Ziegler}, the vertex figure of a $d$-polytope at any vertex $v$ is a $(d-1)$-polytope. Since the vertices of the vertex figure are precisely the intersections of the edges incident to $v$ with the hyperplane containing the vertex figure, this means that the set of edges incident to $v$ cannot be contained in a hyperplane.
\end{proof}

\section{Proofs}\label{sec:proofs}

This section is devoted to the proof of \cref{thm:main}. The idea of the proof is as follows. We first establish what we call the \say{projection lemma} (\cref{lem:proj}) which says that for a $0/1$ polytope $P \subset \Rl^d$, if there exists an orthogonal projection of $P$ to some $k$ coordinates with certain nice properties, then $P$ has good edge expansion. The nice properties that the projection $\pi$ needs to satisfy are that the image of $P$ by $\pi$ is equal to the $k$-dimensional hypercube $C^k$ and that not too many vertices of $P$ are projected to the same vertex of $C^k$. If such a projection exists, we can show that $P$ has good edge expansion by using the fact that the edge expansion of $C^k$ is one. The way this argument works is that given any partition $S \cup T$ of the vertices of $P$, we consider $\pi(S)$ and $\pi(T)$ (which are subsets of the vertices of $C^k$) and use that the edge expansion of $C^k$ is one to show that there are many edges of a certain type in $C^k$. Then, using properties of the projection, we show that all edges of this type lift through $\pi^{-1}$ (i.e. we consider the preimage of each edge by $\pi$) to edges of $P$ that connect a vertex in $S$ to a vertex in $T$. 

After we establish the above \say{projection lemma}, we show using basic probability that for our models of random $0/1$ polytopes, the projection to any $k$ coordinates has the above nice properties with high probability. Here, $k$ is some positive integer that is chosen based on the parameters of the random $0/1$ polytope in question. 

\subsection{The projection lemma}

\begin{lemma}\label{lem:proj}
Let $P \subset \Rl^d$ be a $0/1$ polytope and suppose that there exist $k$
coordinates such that the orthogonal projection $\pi_k$ to those $k$ coordinates satisfies
\begin{enumerate}
    \item $\pi_k P = C^k$. Equivalently, with $C^k$ denoting the $k$-cube in $\pi_k \Rl^d = \Rl^k$, every vertex of $C^k$ appears at least once in $\pi_kV(P)$.
    \item For every vertex $v \in V(C^k)$, the cardinality of $\pi_k^{-1}(v) \cap V(P)$ is at most $c$.
\end{enumerate}
Then the edge expansion of the graph of $P$ is at least $\frac{1}{2c}$.
\end{lemma} 

\begin{proof}
Let $S \subset V(P)$ with $|S| \le |V(P)|/2$ and let $T: = V(P) \setminus S$. Set $s = |S|$. We need to show that there are at least $s/2c$ edges in $P$ which connect a vertex in $S$ to a vertex in $T$. 

Observe that at least one of $\{x \in V(C^k) : \pi_k^{-1}(x) \subset S\}$ or $\{x \in V(C^k) : \pi_k^{-1}(x) \subset T\}$ has cardinality at most $2^{k-1}$. 
Assume that $\{x \in V(C^k) : \pi_k^{-1}(x) \subset S\}$ has cardinality at most $2^{k-1}$. The proof for the other case is nearly the same. 
The projection of $S$, i.e. $\pi_k(S)$, is a subset of vertices of $C^k$ and by assumption 2, $|\pi_k(S)| \ge s/c$. 
There are two cases to consider. 

\textbf{Case 1:} The cardinality of $M:=\pi_k(S) \cap \pi_k(T)$ is at least $s/2c$.

\textbf{Case 2:} The cardinality of $M$ is less than $s/2c$. 

For Case 1, for each $x \in M$, $\pi_k^{-1}(x)$ is a face of $P$ which contains points from $S$ and points from $T$. Since graphs of polytopes are connected, there exists an edge in this face going from a point in $S$ to a point in $T$. Since $|M| \ge s/2c$, we have found $s/2c$ edges in $P$ from $S$ to $T$. Each of these edges is unique because the image of each edge by $\pi_k$ is a unique vertex in $M$.  

For Case 2, let $U = \pi_k(S) \setminus M$. Since $|\pi_k(S)| \ge s/c$ and $|M| \le s/2c$, we have that $|U| \ge s/2c$. 
By assumption 1, we know that $\pi_k(S \cup T)$ contains every vertex of $C^k$. 
Finally, recall that we are assuming that $\{x \in V(C^k) : \pi_k^{-1}(x) \subset S\}$ has cardinality at most $2^{k-1}$. 
This means that $|U| \le 2^{k-1}$.
Now using the fact that the edge expansion of $C^k$ is 1, we know that there are at least $|U|\ge s/2c$ edges of $C^k$ going from a vertex in $U$ to a vertex in $V(C^k)\setminus U$. 
Let $E$ be the set of those edges.
We will show that each such edge lifts (through $\pi_k^{-1}$) to an edge of $P$ that has one point in $S$ and one point in $T$ as its endpoints. 
That is, for each edge $e \in E$, we consider the preimage $\pi_k^{-1}(e)$ and we will show that there exists some edge of $P$ that is contained in $\pi_k^{-1}(e)$ and which has one point in $S$ and one point in $T$ as its endpoints.

For each edge $e \in E$ we have $e = \conv(u,m)$ with $u \in U$ and $m \in V(C^k)\setminus U$. 
The pre-image $\pi_k^{-1}(e)$ is a face (call it $F$) of $P$ which has two proper faces $\pi_k^{-1}(u), \pi_k^{-1}(m)$.\footnote{For those unfamiliar with polytope theory, here is an explanation of why these preimages are faces and/or proper faces: We know that $F$ is a face of $P$ because if $H$ is a hyperplane supporting $e$ as a face of $C^k$, then $\pi_k^{-1}(H)$ is a hyperplane that supports $F$ as a face of $P$. A similar argument shows that $\pi_k^{-1}(u), \pi_k^{-1}(m)$ are both faces of $F$ and they are proper because they do not contain all the vertices of $F$.} 
Since $u \in U$, the face $\pi_k^{-1}(u)$ contains only points from $S$. 
Furthermore, by the way $U$ was constructed, for every $m \in V(C^k)\setminus U$, we know that $\pi_k^{-1}(m)$ contains at least one point from $T$. 
Let $t$ be a point in $T \cap \pi_k^{-1}(m)$. 
We claim that there is an edge in the face $F$ which goes from $t$ to a point $s \in \pi_k^{-1}(u)$. 
Indeed, if this were not the case, all of the edges in $F$ incident to $t$ would be contained in the face $\pi_k^{-1}(m)$. Now if we consider $F$ as a full dimensional polytope in $\aff F$, because $\pi_k^{-1}(m)$ is a proper face of $F$, it is contained in a hyperplane in $\aff F$. This implies that the vertex $t$ has the property that all edges incident to $t$ are contained in a hyperplane in $\aff F$ which is not possible by \cref{prop:poly}.
We have shown that for every edge $e \in E$, there is an edge $e'$ in $P$ which goes from a point in $S$ to a point in $T$ and also that $\pi_k(e') =e$. 
Since all of the edges $e \in E$ are unique, the fact that $\pi_k(e') =e$ for all $e \in E$ implies that all of the edges $e'$ that we construct are unique. Since $|E| \ge s/2c$ we have shown that there are at least this many edges in $P$ going from $S$ to $T$ and we are done. 
\end{proof}

\subsection{The three models of random polytopes} 

In this section we complete the proof of \cref{thm:main}. For the sake of readability, we state \cref{thm:main} separately for each of the three models of random polytopes we consider. We first prove the theorem for the balls-into-bins model $P_n^d$ (\cref{thm:n}). The proof first considers certain \say{degenerate} cases, i.e., when $n$ is either very large or very small. In these cases, it is trivial to show the conclusion of the theorem. For all other cases, we consider the projection of $P_n^d$ to the first $k$-coordinates (for certain $k$ depending on $n$ and $d$) and show that, with high probability, this projection has properties which allow us to obtain the conclusion of the theorem as a direct consequence of \cref{lem:proj}. The proof for the binomial model $P_p^d$ (\cref{thm:p}) is very similar to the one for the balls-into-bins model. Finally, for the uniform model $Q_{n}^d$, instead of redoing the proof a third time, we use a basic result from the theory of random sets to obtain the proof for the uniform model as a direct consequence of the proof for the binomial model (\cref{thm:uniform}).

\begin{theorem}[The balls-into-bins model]\label{thm:n}
Let $S_n^d$ be a set of  $n$ points chosen independently and uniformly from $\{0,1\}^d$. Then the edge expansion of the polytope $P_n^d:= \conv S_n^d$ is at least $1/12d$ with high probability. 
\end{theorem}
\begin{proof}
First, if $n\le d$ then it is clear that $P_n^d$ has edge expansion at least $1/12d$ because $P_n^d$ has at most $d$ vertices. Indeed, given any subset $S\subset P_n^d$ with $|S| \le |P_n^d|/2$, the fact that graphs of polytopes are connected implies that there is at least one edge connecting a vertex in $S$ to a vertex in $P_n^d \setminus S$. Since $|S| \le d/2$, this is enough to show that the edge expansion of $P_p^d$ is at least $2/d \ge 1/12d$. 

If $n \ge d2^d$, then we claim that $P_n^d = C^d$ with high probability. Indeed, the probability that there exists some vertex of $C^d$ that is not chosen once in $S_n^d$ is less than or equal to
$$
2^d\bigg(1-\frac{1}{2^d}\bigg)^{d2^d} \le \bigg(\frac{2}{e}\bigg)^{d} 
$$
and so the probability that $P_n^d \neq C^d$ goes to zero as $d \to \infty$. The fact that $P_n^d$ has edge expansion at least $1/12d$ with high probability now follows from the fact that the edge expansion of $C^d$ is 1. 

Now assume that $d<n<d2^d$. Let $k$ be the largest integer such that $n \ge k2^k$. We will show that by considering the projection of $P_n^d$ to the first $k$ coordinates, we can apply \cref{lem:proj} to show that the edge expansion of $P_n^d$ is at least $1/12d$ with high probability. Note that since $n<d2^d$, this means that $k<d$. Let $\pi_k: V(P_n^d) \to \Rl^k$ denote the orthogonal projection of $V(P_n^d)$ to the first $k$ coordinates. We claim that two things hold with high probability:

\textbf{Claim 1:} Letting $C^k$ denote the $k$-cube in $\pi_k \Rl^d = \Rl^k$, every vertex of $C^k$ appears at least once in $\pi_kV(P_n^d)$. 

\textbf{Claim 2:} For every vertex $v \in V(C^k)$, the cardinality of $\pi_k^{-1}(v)\subset V(P_n^d)$ is at most $6d$.

To prove that the first claim holds with high probability, note that it suffices to prove that every vertex of $C^k$ appears at least once in $\pi_kS_n^d$ with high probability. Observe that $\pi_kS_n^d$ is the same as $S_n^k$. Therefore, the first claim is equivalent to the statement that every vertex of $C^k$ appears at least once in $S_n^k$. Since $n \ge k2^k$, as argued above, we have that the probability that there exists some vertex of $C^k$ that is not chosen once in $S_n^k$ is less than or equal to $(\frac{2}{e})^{k}$. Since we are assuming in this case that $(k+1)2^{k+1} >n >d$, we have that $k \to \infty$ as $d\to \infty$ and therefore the probability that there exists some vertex of $C^k$ that is not chosen once in $S_n^k$  goes to zero as $d \to \infty$. And so we have that that Claim 1 holds with high probability. 

For the second claim, we use the well-known analysis of the classic \say{balls-into-bins} problem from probability theory, see for example \cite{balls}. In our application, that balls are the points in $S_n^d$ and the bins are the vertices of $C^k$. That is, we have $n$ balls each of which is placed into one of $2^k$ bins uniformly at random. Using the fact that $k2^k \le n \le (k+1)2^{k+1}$, by \cite[Theorem 1]{balls}, each bin contains at most $6k$ balls with high probability. Since $k<d$, each bin contains at most $6d$ balls with high probability. In other words, Claim 2 holds with high probability. 

Because Claims 1 and 2 hold with high probability, by \cref{lem:proj} we have that the edge expansion of the graph of $P_n^d$ is at least $1/12d$ with high probability. 
\end{proof}

\begin{theorem}[The binomial model]\label{thm:p}
For any $p \in (0,1)$, let $S_p^d$ be the subset of $\{0,1\}^d$ where each $x \in \{0,1\}^d$ is in $S_p^d$ with probability $p$. Then the edge expansion of the polytope $P_p^d:= \conv S_p^d$ is at least $1/12d$ with high probability.
\end{theorem}
\begin{proof}
First, if $p \le d/2^d$, then it is clear that $P_p^d$ has edge expansion at least $1/12d$ with high probability because it has few vertices with high probability. Indeed, the cardinality of $S_p^d$ is a binomial random variable with number of trials $2^{d}$ and probability of success $p$ and so it has expected value $\mu:=p2^d$. Using the Chernoff bound, we have that $|S_p^d|$ is less than $3\mu = 3p2^d \le3d$ with high probability. This means that $P_p^d$ has at most $3d$ vertices with high probability. Given any subset $S\subset P_p^d$ with $|S| \le |P_p^d|/2$, the fact that graphs of polytopes are connected implies that there is at least one edge connecting a vertex in $S$ to a vertex in $P_p^d \setminus S$. Since $|S| \le 3d/2$, this is enough to show that the edge expansion of $P_p^d$ is at least $2/3d \ge 1/12d$. 

So now assume that $p > d/2^d$. 

Let $k$ be the largest integer such that $p2^d \ge k2^k$. We will show that by considering the projection of $P_p^d$ to the first $k$ coordinates, we can apply \cref{lem:proj} to show that the edge expansion of $P_p^d$ is at least $1/12d$ with high probability. Note that since $p<1$, this means that $k<d$. Let $\pi_k: V(P_p^d) \to \Rl^k$ denote the orthogonal projection of $V(P_p^d)$ to the first $k$ coordinates. We claim that two things hold with high probability:

\textbf{Claim 1:} Every vertex of $C^k$ appears at least once in $\pi_kV(P_p^d)$. 

\textbf{Claim 2:} For every vertex $v \in V(C^k)$, the cardinality of $\pi_k^{-1}(v)\subset P_p^d$ is at most $6d$.

To prove the first claim holds with high probability, observe that for each $v \in V(C^k)$, the set $\pi_k^{-1}(v)$ consists of $2^{d-k}$ vertices of $C^d$. Therefore, the probability that there is some vertex in $C^k$ that doesn't appear in $\pi_kV(P_p^d)$ is equal to $2^k(1-p)^{2^{d-k}}$. Now using the fact that $p2^d \ge k2^k$, we have that $1-p \le 1- \frac{k2^k}{2^d}$. Therefore, the previously mentioned probability is at most $2^k(1-\frac{k}{2^{d-k}})^{2^{d-k}}$. This quantity is less than $(2/e)^k$. Since $k$ is the largest integer such that $p2^d \ge k2^k$, we know that $p2^d \le (k+1)2^{k+1}$. Substituting $p>d/2^d$ into the previous inequality yields $d< (k+1)2^{k+1}$ and so $k \to \infty$ as $d \to \infty$. Now since the probability that there is some vertex that doesn't appear in $\pi_pV(P_n^d)$ is less than $(2/e)^k$, we have that this probability goes to zero as $d \to \infty$ and so Claim 1 holds with high probability.

For the second claim, observe that for each $v \in V(C^k)$, $|\pi_k^{-1}(v)|$ is a binomial random variable with number of trials $2^{d-k}$ and probability of success $p$. This means that the expected value $\mu$ of each of these random variables is $p2^{d-k}$. Now by the fact that $k2^k \le p2^d \le (k+1)2^{k+1}$, we have that $k \le \mu \le 2(k+1)$. This means that $3\mu \le 6(k+1)$. Therefore, using the Chernoff bound, we have
\begin{equation}
\begin{split}
\mathbb{P}(|\pi_k^{-1}(v)| \ge 6(k+1)) &  \le  \mathbb{P}(|\pi_k^{-1}(v)| \ge 3\mu)\\
& \le \bigg(\frac{e^2}{3^3}\bigg)^\mu \\
& \le \bigg(\frac{e^2}{3^3}\bigg)^k.
\end{split}
\end{equation}

This means that the probability that there is some $v \in V(C^k)$ such that $|\pi_k^{-1}(v)| \ge 6(k+1)$ is at most $\bigl(\frac{2e^2}{3^3} \bigr)^k$ which goes to zero as $k \to \infty$. Therefore, with high probability,  for every vertex $v \in V(C^k)$, the cardinality of $\pi_k^{-1}(v)$ is at most $6k+5 < 6d$.

Because Claims 1 and 2 hold with high probability, by \cref{lem:proj} we have that the edge expansion of the graph of $P_p^d$ is at least $1/12d$ with high probability. 
\end{proof}

\begin{theorem}[The uniform model]\label{thm:uniform}
Let $U_{n}^d$ be chosen uniformly at random from the set of all $n$-element subsets of $\{0,1\}^d$. Then the edge expansion of the polytope $Q_{n}^d:= \conv U_{n}^d$ is at least $1/12d$ with high probability. 
\end{theorem}
\begin{proof}
For this proof, we will make use of the fact that the uniform model is in some sense very similar to the binomial model. Recall that in the proof of \cref{thm:p} we showed that with $\pi_k$ denoting the orthogonal projection to the first $k$ coordinates, two claims hold with high probability: 

\textbf{Claim 1:} Every vertex of $C^k$ appears at least once in $\pi_kV(P_p^d)$. 

\textbf{Claim 2:} For every vertex $v \in V(C^k)$, the cardinality of $\pi_k^{-1}(v)\subset P_p^d$ is at most $6d$.

Now it is easy to see that, considering $V(P_p^d)$ as a subset of $\{0,1\}^d$, satisfying Claim 1 and Claim 2 is a \emph{convex property} of subsets of $\{0,1\}^d$ as defined in \cite[Section 1.3]{MR1782847}. Therefore, by \cite[Proposition 1.15]{MR1782847}, Claim 1 and 2 hold with high probability if we replace $P_p^d$ by $Q_{n}^d$ in the statements of these claims. Therefore, again using \cref{lem:proj}, we have that the edge expansion of $Q_{n}^d$ is at least $1/12d$ with high probability. 
\end{proof}

\paragraph{Acknowledgments.}
This material is based upon work supported by the National Science Foundation under Grants CCF-1657939 and CCF-2006994.

\bibliographystyle{abbrv}
\bibliography{bib}

\end{document}